\documentclass[leqno,12pt]{amsart} 

\usepackage{amssymb}

\newtheorem{thm}{Theorem}[section]
\newtheorem{prop}[thm]{Proposition}

\theoremstyle{definition}
\newtheorem{defn}[thm]{Definition}

\theoremstyle{remark}

\numberwithin{equation}{section}

\def\C{\mathbb{C}}
\def\Z{\mathbb{Z}}
\def\R{\mathbb{R}}

\def\Z{\mathbb{Z}}
\def\E{\mathcal{E}}
\def\P{\mathbb{P}}

\def\g{\mathfrak{g}}

\def\O{\mathcal{O}}
\def\D{\mathcal{D}}
\def\F{\mathcal{F}}

\newcommand{\de}{\partial}
\newcommand{\db}{\overline{\partial}}
\newcommand{\ddb}{{\partial }\overline{\partial}}
\newcommand{\pkk}{\lq\lq$p-$K\"ahler\rq\rq}
\newcommand{\kk}{\lq\lq$1-$K\"ahler\rq\rq}
\newcommand{\nkk}{\lq\lq$(n-1)-$K\"ahler\rq\rq}

\def\H{\mathcal{H}}

\begin{document}






\title[Modifications]{Proper modifications of generalized $p-$K\"ahler  manifolds}

\author{Lucia Alessandrini} 

\address
{ Dipartimento di Matematica e Informatica\newline
Universit\`a degli Studi di Parma\newline
Parco Area delle Scienze 53/A\newline
I-43124 Parma
 Italy} \email{lucia.alessandrini@unipr.it}

\subjclass[2010]{Primary 53C55; Secondary 32J27, 32L05}


\keywords{K\"ahler manifold, $p-$K\"ahler manifold, SKT manifold, 
blow-up, modification, balanced manifold.}

\begin{abstract}
 In this paper, we consider  a proper modification $f : \tilde M \to M$ between complex manifolds, and study when a generalized $p-$K\"ahler property goes back from $M$ to $\tilde M$. When $f$ is the blow-up at a point, every generalized $p-$K\"ahler property is conserved, while when
 $f$ is the blow-up along a submanifold, the same is true for $p=1$. For $p=n-1$, we prove that the class of compact generalized balanced manifolds is closed with respect to modifications, and we show that the fundamental forms can be chosen in the expected cohomology class. We get some partial results also in the non-compact case; finally, we end the paper with some examples of  generalized $p-$K\"ahler manifolds.

\end{abstract}

\maketitle

\section{Introduction}

Let $M$  be a complex manifold of dimension $n \geq 2$, let $p$ be an integer, $1 \leq p \leq n-1$.
We shall consider three families of maps, namely:

$\pi_O : \tilde M \to M$, which is the blow-up of $M$ at a point $O$;

$\pi : \tilde M \to M$, which is the blow-up of $M$ along a submanifold $Y$;

$f : \tilde M \to M$, which is a proper modification of $M$ with center $Y$ and exceptional set $E$.

We will study, in this context, when a generalized $p-$K\"ahler property (indicated as  \pkk property, see Definition 2.3) goes back from $M$ to $\tilde M$, or what kind of weaker properties can be obtained. So we unify and generalize analogous results on hermitian symplectic, SKT, balanced, strongly Gauduchon manifolds.
\medskip

The obvious way is to start from a \lq\lq$p-$K\"ahler\rq\rq form $\Omega$ on $M$, and consider the pull-back $f^*\Omega$ on $\tilde M$, which is  \lq\lq closed\rq\rq because $f$ is holomorphic; for the same reason, we get $f^*\Omega \geq 0$, and $f^*\Omega > 0$ on $\tilde M - E$, since 
$f|_{\tilde M - E}$ is biholomorphic.

Nevertheless, strict positivity is not preserved in general: for instance, if $y \in Y,$ and $F := f^{-1}(y)$ is a fibre of dimension $k$, and 
$\omega > 0$ is the $(1,1)-$form of a K\"ahler metric on $M$, it holds 
$\int_F (f^* \omega)^k = 0$.

\medskip
The case when $M$ is K\"ahler  is well known: while $\pi$ (and $\pi_O$) preserve the K\"ahler property, so that $\tilde M$ is K\"ahler too, the famous example of Hironaka (a compact threefold $X$ which is given by a modification of $\P_3$) shows that the K\"ahler property is not preserved by modifications.

Hironaka's example $X$ is a Moishezon manifold (so that all kinds of \lq\lq $p-$K\"ahler\rq\rq proper\-ties are equivalent, except $pK$, see 7.4): this proves that it
 is \lq\lq 2-K\"ahler\rq\rq (i.e. 2K, 2WK, 2S, 2PL, see section 2)  because it is balanced (\cite{AB3}); $X$ is not 1K (nor 1WK, 1S, 1PL) since it contains a curve that bounds.
\medskip

But in general, when we perform a modification of a \kk manifold $M$, it is not guaranteed that $\tilde M$ is regular
(in the sense of Varouchas, see 7.4, i.e. a manifold satisfying the $\ddb$-Lemma). Moreover, it is well-known that if $\tilde M$ satisfies the $\ddb$-Lemma, so does $M$ (see \cite{DGMS}); but it is not known yet
 if a modification of a regular manifold is regular too: this fact sheds further light on the context of the question we stated above.
\medskip

The first result we get (Theorem 3.1) extends the very classical statement: {\it The blow-up at a point of a K\"ahler manifold is a K\"ahler manifold too.} We prove, with a unified proof, that the same holds also for hermitian symplectic, pluriclosed, SKT, balanced, strongly Gauduchon, \dots manifolds: in general, for \pkk manifolds. This result allows one to construct new examples of \pkk manifolds.

Next, in Theorem 3.2, we extend another classical result, that is: {\it If $M$ is a K\"ahler manifold, and $\tilde M$ is obtained from $M$ blowing up a submanifold, then $\tilde M$ is K\"ahler too.} We give a very short proof in the general case of \kk manifolds, which includes also pluriclosed (i.e. SKT) and hermitian symplectic manifolds. The analogous result cannot hold in the generic \pkk case, as we prove by a suitable example.
\medskip

As for compact \nkk manifolds, we complete the study of the invariance of the property  of being \lq\lq balanced\rq\rq with respect to modifications, initiated in \cite{AB6} in the classical case, and due to \cite{Po2} in the sG case: in Theorem 4.1, we prove that a modification $\tilde M$ of a compact \nkk manifold $M$, is \nkk, and in Theorem 4.3 we prove that, when $\tilde M$ is \nkk, then $M$ is \nkk too.
Next we give a partial result in case \pkk.

Here the compactness hypothesis is needed to use the characterization of \pkk 
ma\-nifolds by means of positive currents (see Theorem 2.4). 
But, owing to the use of currents, we lose the link between metrics on $M$ and $\tilde M$: we recapture the link (that is, $f_* \tilde \omega ^{n-1}$ is cohomologous to $ \omega ^{n-1}$) in Proposition 5.1; this result is proved in a more general setting in Theorem 5.2.

\medskip

We look also for another kind of generalization of our main result in \cite{AB5}, i.e., {\it A proper modification $\tilde M$ of a compact balanced manifold $M$ is balanced.} Indeed, we consider non-compact manifolds, but suppose that the center $Y$ is compact. In this case, we can consider Bott-Chern and Aeppli cohomology with compact supports, and use a modified version of the characterization theorem by positive currents; we can prove (see Theorem 6.2) that, under mild cohomological hypotheses, if $M$ is {\it locally balanced with respect to $Y$}, then $\tilde M$ is {\it locally balanced with respect to $E$}.

We end the paper in section 7 with some examples and some remarks on the \lq\lq exactness\rq\rq of the \pkk form.

We would like to thank the Referee for his valuable suggestions.
\bigskip

\section{Preliminaries}

Let $X$ be a complex manifold of dimension $n \geq 2$, let $p$ be an integer, $1 \leq p \leq n-1$; 
 we refer to \cite{HK} (see also \cite{A1}) as regards notation and terminology. 
  To define positivity for forms and currents, let us start from a complex $n-$dimensional euclidean vector space $E$, its associated euclidean vector spaces of $(p,q)-$forms  $\Lambda^{p,q}(E)$ (in particular $E^* = \Lambda^{1,0}(E)$), and a orthonormal basis $\{\varphi_1, \dots, \varphi_n \}$ for $E^*$.
 
 Let us denote $\varphi_I := \varphi_{i_1} \wedge \dots \wedge \varphi_{i_p}$, where $I = (i_1, \dots, i_p)$, $\sigma_p := i^{p^2} 2^{-p}$ and $\Lambda ^{p,p}_{\R} (E^*) := \{ \psi \in \Lambda ^{p,p} (E^*) / \psi = \overline{\psi} \}$. Let $p+k=n$.
 
 We get obviously that
 $\{ \sigma_p \varphi_I \wedge \overline{\varphi_I} , |I| = p \}$ is a orthonormal basis for $\Lambda ^{p,p}_{\R} (E^*)$, and $$dV = (\frac{i}{2}  \varphi_1 \wedge \overline{\varphi_1}) \wedge \dots \wedge (\frac{i}{2}  \varphi_n \wedge \overline{\varphi_n}) = \sigma_n \varphi_I \wedge \overline{\varphi_I} , \ I=(1, \dots , n)$$
is a volume form. 
\medskip

 \begin{defn} 
\begin{enumerate}

\item A $(n,n)-$form $\tau$ is called {\it positive} ({\it strictly positive}) if $\tau = c\ dV$ with $c \geq 0 \ ( c>0)$.

\item $\eta \in \Lambda^{p,0} (E^*)$ is called {\it simple} (or decomposable) if and only if there are $\{\psi_1, \dots, \psi_p \} \in \Lambda^{1,0}(E)$ such that $\eta = \psi_{1} \wedge \dots \wedge \psi_{p}$.

\item $\Omega \in \Lambda ^{p,p}_{\R} (E^*)$ is called {\it strongly positive} ($\Omega \in SP^p$) if and only if 
$ \Omega = \sigma_p \sum \eta_j \wedge \overline{\eta_j} ,$ with $\eta_j$ simple.

\item $\Omega \in \Lambda ^{p,p}_{\R} (E^*)$ is called {\it weakly positive} ($\Omega \in WP^p$) if and only if 
for all $\{\psi_1, \dots, \psi_m \} \in \Lambda^{1,0}(E)$, and for all $I = (i_1, \dots, i_k)$ with 
$k+p=n$,
$\Omega \wedge \sigma_k \psi_I \wedge \overline{\psi_I}$ is a positive $(n,n)-$form. It is called {\it transverse} when it is strictly weakly positive, i.e. when $\Omega \wedge \sigma_k \psi_I \wedge \overline{\psi_I}$ is a strictly positive $(n,n)-$form.
\end{enumerate}
\end{defn}
\medskip

 {\bf Remarks.}
 \medskip
 
 a) There is also an intermediate \lq\lq natural\rq\rq definition of positivity, given in terms of eigenvalues, or al follows: 
 \lq\lq$\Omega \in \Lambda ^{p,p}_{\R} (E^*)$ is positive ( $\Omega \in P^p$) if and only if for every $\eta \in \Lambda^{k,0} (E^*)$, $(k+p=n)$,
$ \Omega \wedge \sigma_k \eta \wedge \overline{\eta}$ is a positive $(n,n)-$form.\rq\rq (see \cite{HK}, Theorem 1.2). 
 \medskip
 
 b) Positive forms (as in a)) are not considered by Lelong (\cite{Le}) nor by Demailly (\cite{De}); both them call positive forms (this is the \lq\lq classical sense\rq\rq) what we call weakly positive forms. The strongly positive forms are called {\it decomposable} by Lelong.
 \medskip
 
 c) The sets $P^p, SP^p, WP^p$ and their interior parts are indeed cones; moreover, there are obvious inclusions: 
 $$ SP^p \subseteq P^p \subseteq WP^p \subseteq \Lambda ^{p,p}_{\R} , \ \ (SP^p)^{int} \subseteq (P^p)^{int} \subseteq (WP^p)^{int}.$$
 
 d) When $p=1$ or $p=n-1$, the cones coincide, since every $(1,0)-$form is simple (and hence also every $(n-1,0)-$form is simple).
 \medskip 
 
 e) In the intermediate cases, $1< p< n-1$, the inclusions are strict: indeed, if $\{\varphi_1, \dots, \varphi_4 \}$ is a basis for $\Lambda ^{1,0} (\C^4)$, then it is easy to prove that $\varphi_1 \wedge \varphi_2 + \varphi_3 \wedge \varphi_4$ is not a simple $(2,0)-$form; moreover, in $\C^n$, $(\varphi_1 \wedge \varphi_2 + \varphi_3 \wedge \varphi_4) \wedge \varphi_5 \wedge \dots \wedge \varphi_{p+2}$ is not a simple $(p,0)-$form, for $p > 2$. 
 
 By Proposition 1.5 in \cite{HK}, this implies that
 $(\varphi_1 \wedge \varphi_2 + \varphi_3 \wedge \varphi_4) \wedge (\overline{\varphi_1 \wedge \varphi_2 + \varphi_3 \wedge \varphi_4})$ is a positive $(2,2)-$form which is not strongly positive.
 
Moreover,  the authors exhibit a $(p,p)-$form which is in the interior of the cone $WP^p$, but has a negative eigenvalue, so that it does not belong to the cone $P^p$.
 \medskip
 
 f) Duality: It is not hard to prove that, for $p+k=n$:
 $$\Omega \in WP^p \iff \forall \ \Psi \in SP^k, \Omega \wedge \Psi \geq 0, \ \quad \Omega \in P^p \iff \forall \ \Psi \in P^k, \Omega \wedge \Psi \geq 0.$$
 \medskip

Let us go back to manifolds: we denote  by ${\D}^{p,p}(X)_\R$ the space of compactly supported real $(p,p)-$forms on $X$ and by ${\E}^{p,p}(X)_\R$ the space of real $(p,p)-$forms on $X$. 

Their dual spaces are: ${\D}_{p,p}'(X)_{\R}$ (also denoted by ${\D '}^{k,k}(X)_{\R}$, where $p+k=n$), the space of real currents of bidimension $(p,p)$ or bidegree $(k,k)$, which we call $(k,k)-$currents, and 
${\E}_{p,p}'(X)_{\R}$ (also denoted by ${\E '}^{k,k}(X)_{\R}$), the space of compactly supported real $(k,k)-$currents on $X$. 

We shall denote by $[Y]$ the current given by the integration on the irreducible analytic subset $Y$. 
\medskip

We shall define weakly positive, positive, strongly positive currents (see f.i. \cite{HK}). For simplicity, let $N$ be a {\it compact} $n-$dimensional manifold, and $1 \leq p \leq n-1$. 
\medskip

\begin{defn}  
\begin{enumerate}
\item $\Omega \in {\E}^{p,p}(N)_\R$ is called strongly positive (resp. positive, weakly positive, transverse or strictly weakly positive) if $\forall \ x \in N, \ \Omega_x \in SP^p (T'_xN)$ (resp. $P^p (T'_xN), \  WP^p (T'_xN),$ $(WP^p (T'_xN))^{int}$). 

These spaces of forms are denoted by
$SP^p(N), \ P^p(N),$ $ WP^p(N),$ $(WP^p(N))^{int}$. 

\item Let $T \in {\D}_{p,p}'(N)_{\R}$ be a current of bidimension $(p,p)$ on $N$. Then we have:

weakly positive currents: $T \in WP_p(N) \iff T(\Omega) \geq 0\ \ \forall \ \Omega \in SP^p(N)$. 

positive currents: $T \in P_p(N) \iff T(\Omega) \geq 0\ \ \forall \ \Omega \in P^p(N)$. 

strongly positive currents: $T \in SP_p(N) \iff T(\Omega) \geq 0\ \ \forall \ \Omega \in WP^p(N)$. 
\end{enumerate}
\end{defn}
\medskip

 {\bf Remarks.}  There are obvious inclusions between the previous cones of currents, that is, 
 $SP_p(N) \subseteq  P_p(N) \subseteq WP_p(N)$. Demailly (\cite{De}, Definition 1.3) does not consider $P_p(N)$, and indicates $WP_p(N)$ as the cone of positive currents; there are no uniformity of notation in the papers of Alessandrini and Bassanelli.
 \medskip
 
We shall need de Rham cohomology, and also Bott-Chern and  Aeppli cohomology (for which the notation is not standard, so that we recall it below): they can
be described using forms or currents of the same bidegree: 

$$H_{dR} ^{k,k}(X, \R) :=\frac{\{ \varphi \in {\E}^{k,k}(X)_\R;
d\varphi =0\}}{\{d\psi ;\psi \in {\E}^{2k-1}(X)_\R\}}\simeq\frac{\{T \in
{\D '}^{k,k}(X)_{\R}; dT =0\}}{\{dS ; S \in  {\D'}^{2k-1}(X)_{\R}
\}}$$
$$H_{\ddb}^{k,k}(X, \R) = \Lambda_\R ^{k,k}(X) = H_{BC}^{k,k}(X, \R) :=\frac{\{ \varphi \in {\E}^{k,k}(X)_\R;
d\varphi =0\}}{\{i\partial\overline{\partial}\psi ;\psi \in {\E}^{k-1,k-1}(X)_\R\}}\simeq$$
$$\simeq\frac{\{T \in
{\D '}^{k,k}(X)_{\R}; dT =0\}}{\{i\partial\overline{\partial}A ; A \in  {\D '}^{k-1,k-1}(X)_{\R}
\}}$$
$$H_{\de + \db}^{k,k}(X, \R) =V_\R ^{k,k}(X) = H_{A}^{k,k}(X, \R) :=\frac{\{ \varphi \in {\E}^{k,k}(X)_\R;
i\ddb\varphi =0\}}{\{\varphi = \de \overline\eta + \db \eta ; \eta \in {\E}^{k,k-1}(X)\}} \simeq$$
$$\simeq\frac{\{T
\in {\D '}^{k,k}(X)_{\R}; i\ddb T =0\}}{\{ \de \overline S + \db S ; S \in  {\D '}^{k,k-1}(X)
\}}.$$

\medskip
In general when the class of a current vanishes in one of the previous cohomology groups, we say that the current \lq\lq bounds\rq\rq or is \lq\lq exact\rq\rq. 
\medskip

We collect what we called in the Introduction  {\bf \lq\lq $p-$K\"ahler\rq\rq properties} in the following definition (see \cite{A1}, and also the next Remarks).
 
 \begin{defn} Let $X$ be a complex manifold of dimension $n \geq 2$, let $p$ be an integer, $1 \leq p \leq n-1$.
\begin{enumerate}

\item $X$ is a $p-$K\"ahler ({\bf pK}) manifold if it has a closed transverse  $(p,p)-$form $\Omega$. 

\item $X$ is a weakly $p-$K\"ahler ({\bf pWK}) manifold if it has a transverse $(p,p)-$form $\Omega$ with $\de \Omega = \ddb \alpha$ for some form $\alpha$.

\item $X$ is a $p-$symplectic ({\bf pS}) manifold if it has a closed transverse  real $2p-$form $\Psi$; that is, $d \Psi = 0$ and $\Omega := \Psi^{p,p}$ (the  $(p,p)-$component of $\Psi$) is transverse.

\item $X$ is a $p-$pluriclosed  ({\bf pPL}) manifold if it has a transverse $(p,p)-$form $\Omega$ with $\ddb \Omega = 0.$
\end{enumerate}
\end{defn}
\medskip

Notice that:
$pK \Longrightarrow pWK  \Longrightarrow pS  \Longrightarrow pPL;$ 
as regards examples and differences under these classes of manifolds, see \cite{A1}.

When $X$ satisfies one of these definitions, in the rest of the paper we will call it generically a {\bf \lq\lq$p-$K\"ahler\rq\rq manifold}; the form $\Omega$, called a  {\bf \pkk form}, is said to be {\lq\lq closed\rq\rq}. This may be a little bit worrying to read, but the benefit is that we do not write a lot of similar proofs.
\medskip

{\bf Remarks.} For $p=1$, a transverse form is the fundamental form of a hermitian metric, so that we can consider $1-$K\"ahler (i.e. K\"ahler), weakly $1-$K\"ahler,  $1-$symplectic, $1-$pluriclosed {\it metrics}. 
 $1-$symplectic manifolds are also called 
{\it hermitian symplectic} (\cite{ST}).

In \cite{Eg}, pluriclosed (i.e. $1-$pluriclosed) metrics are defined (see also \cite{ST}), while in  \cite{FT1} a 1PL metric (manifold) is called a {\it strong K\"ahler metric (manifold) with torsion} (SKT). 
\medskip

 For  $p = n-1$, we get a hermitian metric too, because  every transverse $(n-1,n-1)-$form $\Omega$ is in fact given by $\Omega = \omega^{n-1}$, where $\omega$ is a transverse $(1,1)-$form (see f.i. \cite{Mi}, p. 279). 
This case was studied by Michelsohn  in \cite{Mi}, where $(n-1)-$K\"ahler manifolds are called {\it balanced} manifolds. 

Moreover, $(n-1)-$symplectic manifolds are called {\it strongly Gauduchon manifolds (sG)} by Popovici (compare Definition 2.3 (3)  and Theorem 2.4 (3) with  \cite{Po1}, Definition 4.1 and Propositions 4.2 and 4.3; see also \cite{Po2}), while $(n-1)-$pluriclosed metrics are called {\it standard} or {\it Gauduchon metrics}. Recently, weakly $(n-1)-$K\"ahler manifolds have been called {\it superstrong Gauduchon, (super sG)} (\cite{PU}).
\bigskip

In the case of a {\it compact} manifold $N$, we got the following characterization (see \cite{A1}, Theorems 2.1, 2.2, 2.3, 2.4)

\begin{thm} 

\begin{enumerate}

\item Characterization of compact $p-$K\"ahler (pK) manifolds.

$N$ has a strictly weakly positive (i.e. transverse) $(p,p)-$form $\Omega$ with $\de \Omega = 0$,  if and only if $N$ has no strongly positive currents $T \neq 0$, of bidimension $(p,p)$, such that $T = \de  \overline S + \db S$ for some current $S$ of bidimension $(p,p+1)$ (i.e.  $T$  \lq\lq bounds\rq\rq in $H_{\de + \db}^{k,k}(N)$, i.e. $T$ is the $(p,p)-$component of a boundary).

\item Characterization of compact weakly  $p-$K\"ahler (pWK) manifolds.

$N$ has a strictly weakly positive $(p,p)-$form $\Omega$ with $\de \Omega = \ddb \alpha$ for some form $\alpha$,  if and only if $N$ has no strongly  positive currents $T \neq 0$, of bidimension $(p,p)$, such that $T = \de  \overline S + \db S$ for some current $S$ of bidimension $(p,p+1)$ with $\ddb S = 0$ (i.e.  $T$  is closed and \lq\lq bounds\rq\rq in $H_{\de + \db}^{k,k}(N)$). 

\item Characterization of compact $p-$symplectic (pS) manifolds.

$N$ has a real $2p-$form $\Psi = \sum_{a+b=2p} \Psi^{a,b}$, such that $d \Psi = 0$ and the 
$(p,p)-$form $\Omega := \Psi^{p,p}$ is 
strictly weakly positive,   if and only if $N$ has no strongly  positive currents $T \neq 0$, of bidimension $(p,p)$, such that $T = d S$ for some current $S$  (i.e.  $T$  is a boundary with respect to de Rham cohomology). 

\item Characterization of  compact $p-$pluriclosed (pPL) manifolds.

$N$ has a strictly weakly positive $(p,p)-$form $\Omega$ with $\ddb \Omega = 0$,  if and only if $N$ has no strongly  positive currents $T \neq 0$, of bidimension $(p,p)$, such that $T = i \ddb A$ for some current $A$ of bidimension $(p+1,p+1)$ (i.e.  $T$  \lq\lq bounds\rq\rq in $H_{\ddb}^{k,k}(N)$). 
\end{enumerate}

 \end{thm}

\medskip
{\bf Remark.} Every compact complex manifold supports Gauduchon metrics, that is, is $(n-1)$PL: in fact, by Theorem 2.4 (4), if 
$T$ is a strongly  positive $(1,1)-$current, such that $T = i \ddb A$, $A$ turns out to be a plurisubharmonic function; but $N$ is compact, so that $A$ is constant, and $T=0$.
\medskip

Lastly, let us recall a Support Theorem, which we shall frequently use for $p=n-1$.

\begin{thm} {\rm (see \cite{AB4}, Theorem 1.5)} Let $X$ be a $n-$dimensional complex manifold, $E$ a compact analytic subset of $X$; call $\{ E_j \}$ the irreducible components of $E$ of dimension $p$. Let $T$ be a weakly positive $\ddb-$closed current of bidimension $(p,p)$ on $X$ such that supp $T \subseteq E$. Then there exist $c_j \geq 0$ such that $S := T - \sum_j c_j [E_j]$ is a weakly positive $\ddb-$closed current of bidimension $(p,p)$ on $X$,  
supported on the union of the irreducible components of $E$ of dimension bigger than $p$.
\end{thm}
\bigskip

\section{Blow-up of manifolds}\label{S3}

Let $M$  be a connected complex manifold, with $n  = dim M \geq 2$, let $p$ be an integer, $1 \leq p \leq n-1$.
As said in the Introduction, we shall consider three kinds of proper modifications:

$\pi_O : \tilde M \to M$, which is the blow-up of $M$ at a point $O$;

$\pi : \tilde M \to M$, which is the blow-up of $M$ along a compact submanifold $Y$;

and an arbitrary  proper modification of $M$ with compact center $Y$, $f : \tilde M \to M$.

Recall that a complex manifold $\tilde M$ together with a proper holomorphic map $f : \tilde M \to M$ is called a (smooth) proper modification of $M$ if there is a thin set $Y$ in $M$ such that $f^{-1}(Y)$ is thin in $\tilde M$, and the restricted map $f$ from $\tilde M - f^{-1}(Y)$ to $M-Y$ is  biholomorphic. 

Grauert and Remmert (see \cite{GR2}, pages 214-215) proved among others that $Y$ can be chosen as an analytic set of codimension $\geq 2$ such that 
$E := f^{-1}(Y)$ is an analytic set of pure codimension one in $\tilde M$, called the exceptional set of the modification.

\medskip

We will study, in this context, when a \pkk property goes back from $M$ to $\tilde M$.
The problem is completely solved for $\pi_O$ by Theorem 3.1, for which we give a proof by direct computation, that unifies all \pkk cases, some of which are well known when $p=1$.

\begin{thm} Let $\pi_O : \tilde M \to M$ be the blow-up of $M$ at a point $O$; for every $p, 1 \leq p \leq n-1$, whenever $M$ is \pkk,  $\tilde M$ is also  \pkk .
\end{thm}

\begin{proof} First of all, let us recall the classical proof for K\"ahler manifolds.
Let us choose coordinates $\{ z_j \}$ around $O \in M$, such that on $U_{2 \epsilon} :=\{ ||z|| < 2 \epsilon \}$ and
$\tilde U_{2 \epsilon} := \pi_O^{-1} (U_{2 \epsilon})$, $\pi_O$ is nothing but the blow-up of $\C^n$ at 0, with exceptional set $E := \pi_O^{-1}(0) \simeq \P_{n-1}$.

With obvious notation, consider a cut-off function $\chi \in C_0^{\infty} (U_{2 \epsilon})$, $\chi = 1$ on $U_{\epsilon}$, and put, for $x \in \tilde U_{2 \epsilon}$,
$$\tilde \theta_x := i \ddb (\chi (\pi_O (x)) log || x ||^2),$$ 
 where $i \ddb ( log || x ||^2)$ is just the pull-back of the Fubini-Study $(1,1)-$form on $\P_{n-1}$ under the map $j: \tilde U_{2 \epsilon}
\to \P_{n-1}$ which is the identity on $\P_{n-1}$ and maps every $x \in \tilde U_{2 \epsilon} - \P_{n-1}$ to the line $[x] \in \P_{n-1}$ that passes through $x$ (see \cite{GH}, page 186).
\medskip

The form $\tilde \theta$ turns out to be a global closed real $(1,1)-$form, with {\it supp} $\tilde \theta \subset \tilde U_{2 \epsilon}$; moreover, 
$\tilde \theta \geq 0$ on $\tilde U_{\epsilon}$.
For $x \in E$, $\tilde \theta_x > 0$ only on vectors in $T_x'E$;  $\tilde \theta$
is not strictly positive on $E$.

Nevertheless, starting from a K\"ahler form $\Omega$ on $M$, we can consider
$\pi_O^* \Omega$ which is a  closed real $(1,1)-$form on $\tilde M$, with  
$\pi_O^* \Omega \geq 0$ and $(\pi_O^* \Omega)_x (v \wedge \overline v) > 0$ when $x \in E$ and $v \in T_x' \tilde M$ is orthogonal to $T_x'E$.

Moreover, for  $x$ in the closure of  $ \tilde U_{2 \epsilon} - \tilde U_{\epsilon}$, the values of $(\pi_O^* \Omega)_x$ on positive $(1,1)-$vectors $v \wedge \overline v$ have a positive lower bound.
Hence there is a $c > 0$ such that 
$\tilde \Omega := \pi_O^* \Omega + c \tilde \theta$ is a  K\"ahler form for $\tilde M$.
\medskip

Notice that this proof (the classical one) also works for \kk manifolds; indeed, the summand $c \tilde \theta$ is $d-$closed, and hence also 
\lq\lq closed\rq\rq (see Definition 2.3).
\medskip

On the contrary, in the generic \pkk case, starting from a \pkk form $\Omega$ on $M$, $\tilde \Omega := \pi_O^* \Omega + c \tilde \theta^{p}$ is not strictly weakly positive on $E$, because when $p > 1$, in a $p-$vector $X = v_1 \wedge \dots \wedge v_p$ it is possible to have, for instance, $v_1 \in (T_x'E)^{\perp}$, $v_2, \dots , v_p  \in T_x'E$, so that both summands vanish on the strictly positive $(p,p)-$vector $X \wedge \overline X$.
\medskip

When $p > 1$, we can argue as follows. Let us consider the standard K\"ahler form on $U_{2 \epsilon}$, i.e. 
$$\omega =i \ddb ||z||^2 = \frac{i}{2} \sum_j dz_j \wedge d \overline z_j.$$
Put  $\Theta := \pi_O^* \omega \wedge \tilde \theta^{p-1} + \tilde \theta^{p}$.

{\bf Claim.} The form $ \Theta$ is a  closed real $(p,p)-$form on $\tilde M$, with {\it supp} $ \Theta \subset \tilde U_{2 \epsilon}$; moreover, 
$\Theta \geq 0$ on $\tilde U_{\epsilon}$ and $\Theta > 0$ on $E$.

This is a local construction, based only on the geometry of the blow-up $\pi_O$.
\bigskip

Now, let $\Omega$ be a \pkk form for $M$; the following claim is clear.

{\bf Claim.} $\pi_O^* \Omega$ is a  \lq\lq closed\rq\rq real $(p,p)-$form on $\tilde M$, with  
$\pi_O^* \Omega \geq 0$; moreover, for  $x$ in the closure of  $ \tilde U_{2 \epsilon} - \tilde U_{\epsilon}$ , the values of $(\pi_O^* \Omega)_x$ on positive $(p,p)-$vectors have a positive lower bound.
\medskip

Hence, there is a $c > 0$ such that 
$\tilde \Omega := \pi_O^* \Omega + c \Theta$ is a  \lq\lq closed\rq\rq transverse $(p,p)-$form on $\tilde M$, that is, a \pkk form for $\tilde M$. 
\end{proof}
\bigskip

The case 1PL (where $\tilde \Omega$ is simply $ \pi_O^* \Omega + c \tilde \theta$) was proved in \cite{FT1}, 3.1. The authors proved also, using a similar technique,  the persistence of the 1PL property for a blow-up $\pi$ along a submanifold, as in the classical 1K case (3.2 ibidem).
Let us give here a simpler proof, which includes all \lq\lq 1-K\"ahler\rq\rq cases, by using the fact that $\pi$ is a projective morphism. 
\medskip

Recall that a blow-up is a projective morphism, hence it is a {\it K\"ahler morphism} (in the sense of \cite{F}, Definition 4.1; recall also \cite{V4}, pp. 23-24); this means that there is an open covering $\{U_j\}$ of $\tilde M$, and, for every $j$, smooth functions $p_j : U_j \to \C$ such that:

$\forall y \in M,$ the restriction of  $p_j$  to $U_j \cap \pi^{-1}(y)$ is strictly plurisubharmonic , and 

$p_j - p_k$ is pluriharmonic on $U_j \cap U_k$.

This gives a {\it relative K\"ahler form} $\tilde \beta$ for $\pi$, that is,  $\tilde \beta := i \ddb p_j$ on $U_j$ gives a globally defined real closed $(1,1)-$form, strictly positive on the fibres
(but notice that the $(1,1)-$form $\tilde \beta$ may not be $\geq 0$ in all directions).

\begin{thm} Let $\pi : \tilde M \to M$ be the blow-up of $M$ along a compact submanifold $Y$; if $M$ is \kk, then $\tilde M$ is \kk too.
\end{thm}

\begin{proof} Following \cite{F}, Lemma 4.4, choose a \kk form $\omega$ for $M$; since $Y$ is compact, there is a constant $C >0$ such that $\tilde \omega := \tilde \beta + C \pi^* \omega > 0$; since $\tilde \beta$ is $d-$closed, $\tilde \omega$ turns out to be \lq\lq closed\rq\rq. 
\end{proof}
\medskip

{\bf Remark.} Example 7.3 in section 7 proves that Theorem 3.2 cannot hold for a generic $p >1$; the case $p = n-1$ is discussed in the next section, for compact manifolds, and in section 6 for non-compact manifolds.
\bigskip

\section{Modifications of compact manifolds}

While we cannot use the previous proof in the case $p>1$, nor one similar to that of Theorem 3.1, on {\it compact} manifolds we can also solve the case $p = n-1$ for arbitrary modifications, as done in \cite{AB4}, Theorem 2.4 in case K and in \cite{Po2} in case S.

\begin{thm} Let $M, \tilde M$ be compact $n-$dimensional manifolds, let $f : \tilde M \to M$ be a modification with center $Y$ (an analytic set of codimension $\geq 2$) and exceptional set $E$, whose ($(n-1)-$dimensional) irreducible components are $\{E_j\}$. If  $M$ is \nkk, then $\tilde M$ is \nkk 
too.
\end{thm}

\begin{proof}
Notice that every compact complex $n-$dimensional manifold is $(n-1)PL$, as we pointed out in section 2.

Let $T \geq 0$ be an \lq\lq exact\rq\rq $(1,1)-$current  on $\tilde M$, as stated in the Characterization Theorem 2.4.
Since $f_*T$ has the same properties on $M$, we get $f_*T=0$, which implies that {\it supp} $T \subseteq E$, and more precisely
$T= \sum c_j [E_j], \ c_j \geq 0$, by the Support Theorem 2.5.
Therefore $T=0$ by the following Proposition (which is  more general, since $T$ is not supposed to be positive and $M, \tilde M$ are not compact).
\end{proof}

\begin{prop} Let $M, \tilde M$ be  $n-$dimensional manifolds, let $f : \tilde M \to M$ be a proper modification with compact center $Y$ (an analytic set of codimension $\geq 2$) and exceptional set $E$, whose ($(n-1)-$dimensional) irreducible components are $\{E_j\}$.  Let $R= \sum c_j [E_j], \ c_j \in \R$; $R$ is a closed real $(1,1)-$cur\-rent on $\tilde M$.
The following statements are equivalent:

\begin{enumerate}
\item  $R$ is the component of a boundary, i.e. its class vanishes in $H_{\de + \db}^{1,1}(\tilde M)$;

\item $R$ is  a boundary, i.e. its class vanishes in $H_{dR} ^{1,1}(\tilde M, \R)$;

\item $R$ is  $\ddb$-exact, i.e. its class vanishes in $H_{\ddb}^{1,1}(\tilde M)$;

\item $c_j = 0 \ \forall j$,  i.e. $R=0$.

\end{enumerate}

\end{prop}

\begin{proof} The implications $(4) \Rightarrow (3) \Rightarrow (2) \Rightarrow (1)$ are  obvious.

$(1) \Rightarrow (2)$: see Lemma 8 in \cite{AB2}, where the hypothesis is: $R$ is a closed $(1,1)-$current on $\tilde M$ such that $f_*R=0$ 
(notice that if $R= \sum c_j [E_j], \ c_j \in \R$, then $f_*R=0$, because ${\rm codim} Y \geq 2$).
We recall here the proof.

Let  $R = \de  \overline S + \db S$ for some $(1,0)-$current $S$; since $R$ is closed,  we get $\ddb S =0$.
Consider $\de S$: it is a $\db -$closed $(2,0)-$current, hence it is a holomorphic 2-form on $\tilde M$; the same holds for $\de (f_*S)$ on $M$. 

Since $\de (f_*S)$ is smooth and $\de -$exact, we can find a $(1,0)-$form $\varphi$ and a distribution $t = a+ib$ on $M$ such that
$f_*S = \varphi + \de t = \varphi + \de a + i \de b$.
\medskip

The explanation is the following (see section 2): consider the isomorphism $j$ (induced by the identity) between smooth and non-smooth (i.e. involving currents) cohomology: for instance,  $j$ maps the class $[\gamma]$ of a smooth $\de$-closed $(2,0)$-form $\gamma$ in the cohomology space $H_\de ^{2,0} (M)$, to the class $\{ \gamma \}$, in the cohomology space of currents (denoted for the moment by $K_\de ^{2,0} (M)$).
Call $\alpha$ the holomorphic 2-form $\de (f_*S)$ on $M$. Since $\de \alpha =0$, $[ \alpha ] \in H_\de ^{2,0} (M)$; but by definition, $\{ \alpha \}=0 \in 
K_\de ^{2,0} (M)$, thus $0 = [ \alpha ] \in H_\de ^{2,0} (M)$, so that $\alpha = \de \mu$ for some smooth 1-form $\mu$.
Therefore $\de (f_* S - \mu)=0$, hence $\{ f_* S - \mu \} \in 
K_\de ^{1,0} (M) \simeq H_\de ^{1,0} (M)$, that is, there is a smooth form $\nu$ such that $\{ f_* S - \mu \} = \{ \nu \}$, i.e. there is a distribution $t$ such that 
$f_* S - \mu  =  \nu + \de t$, as stated.
\medskip

Now we use $f_*R=0$ as follows:
$$0= f_*(\de  \overline S + \db S) = \   \de  (\overline \varphi + \overline {\de a} - i \overline {\de b}) + \db (\varphi + \de a + i \de b) = \de  \overline \varphi  + \db \varphi - 2i \ddb b. $$
Thus $\ddb b$ is smooth, hence also $b$ is smooth, and we can pull it back to $\tilde M$.

Define $s := S - f^*(\varphi + i \de b)$; we get
$$\db s + \de \overline s = \db (S - f^*(\varphi + i \de b)) + \de (\overline S - f^*(\overline \varphi - i \db b)) =$$
$$ \db S + \de \overline S - f^*(\db \varphi - i \ddb b) -  f^*( \de \overline \varphi - i \ddb b))=$$
$$ R - f^*(\db \varphi +  \de \overline \varphi - 2i \ddb b) = R;$$
moreover 
$$\de s = \de S - \de f^*\varphi = \de S -  f^*(\de (f_* S));$$
both summands are holomorphic 2-forms on $\tilde M$, and they coincide outside the exceptional set $E$: therefore they coincide, hence $\de s = 0$. Thus $R= d(s + \overline s)$ is a boundary.
\medskip

$(2) \Rightarrow (3)$: Let $R= dQ = \de \overline S + \db S$, for a real 1-current $Q = Q^{0,1} + Q^{1,0} =\overline S + S$, where $S$ is a $\de -$closed $(1,0)-$current. As before, $0 = f_*R = d(f_*Q)$, so that we can choose a smooth representative of the cohomology class of the $d-$closed 1-current $f_*Q$ on $M$; that is, $f_*Q = \varphi + da$, where $\varphi$ is a smooth closed 1-form and $a$ is a distribution on $M$.

Let $q := Q -f^*\varphi$; it holds $d q = d Q -f^*d \varphi = d Q = R$ and, as regards the $(0,1)-$part,  

$f_* q^{0,1} = f_*(Q^{0,1} -f^*\varphi^{0,1})  = f_*Q^{0,1} - \varphi^{0,1} = \db a.$

Since $R$ is a $(1,1)-$current, $\db q^{0,1} =0$, so it represents a class in $H_{\db}^{0,1} (\tilde M) \simeq 
H_{\db}^{0,1} (M)$ (a classical result), but this class vanishes in $M$, because $f_* q^{0,1} = \db a$; thus it vanishes in $\tilde M$, i.e. $q^{0,1} = \db b$.

Hence 
$$R = d q = \de q^{0,1}  + \db \overline{q^{0,1}}  = \ddb (b - \overline b).$$

$(3) \Rightarrow (4)$: 
Suppose $R = i \ddb a$: since $f_* R =  0$, $f_*a$ is pluriharmonic on $M$ (there is a smooth pluriharmonic function $h$ such that $f_* a = h$ a.e.). Hence $f^* h = h \circ f$ is pluriharmonic on $\tilde M$, so that $R= i \ddb (a - f^* h)$, where the distribution $a - f^* h$ is supported on $E$, because $f|_{\tilde M - E}$ is a biholomorphism. 

Let $x$ be a smooth point $x \in E$ (as a matter of fact, $x \in E_k$ for some $k$); choose a neighborhood $U$ of $x$ with coordinates $\{w_j\}$ such that, in $U$, $R = c_k[E_k] = i c_k\pi^{-1} \ddb log ||w_n||;$
thus in $U$ the distribution 
$$  i c_k\pi^{-1} log ||w_n|| - (a - f^* h)$$
is pluriharmonic, hence smooth. This implies that 
$a - f^* h$, which is a distribution supported on $E$, vanishes in $U$. We conclude in this manner that $R= \sum c_j [E_j] =  i \ddb (a - f^* h) = 0$.
\end{proof}

\medskip

More than that, we can prove that the class of compact \nkk manifolds is closed with respect to modifications.

\begin{thm} Let $M, \tilde M$ be compact $n-$dimensional manifolds, let $f : \tilde M \to M$ be a modification. If  $\tilde M$ is \nkk, then $M$ is \nkk too.
\end{thm}

\begin{proof}
The case $(n-1)PL$ is obvious. The case $(n-1)K$ is proved in \cite{AB6}, the case $(n-1)S$ is proved by Popovici in \cite{Po2}; the proofs are similar, nevertheless, as the author says in the Introduction, the arguments are considerably simplified by the fact that one can handle \lq\lq pull-back\rq\rq of $d-$closed positive $(1,1)-$currents by their local potentials.
Let us consider here the WK-case, to complete the proof of the Theorem.
\medskip

Take a $(1,1)-$current $T \geq 0$ on $M$, such that $dT=0$ and $T = \de  \overline S + \db S$. Consider the following result:

\begin{thm} {\rm (Theorem 3 in \cite{AB6})} Let $M, \tilde M$ be complex manifolds, and let $f : \tilde M \to M$ be a proper modification. Let $T$ be a positive $\ddb-$closed $(1,1)-$current on $M$. Then there is a unique positive $\ddb-$closed $(1,1)-$current $\tilde T$ on $\tilde M$ such that $f_* \tilde T = T$ and $\tilde T \in f^* \{T\} \in H_{\de + \db}^{1,1}(\tilde M, \R).$ 
\end{thm}

 Looking carefully through the details of the proof (see also Theorem 3.9 and Proposition 3.10 in \cite{AB5}), it is not hard to notice that, when $T$ is $d-$closed, $\tilde T$ becomes $d-$closed too (in the estimates, this is the \lq\lq classical case\rq\rq).
 
 Thus, in our situation, $\tilde T$ is a closed positive  $(1,1)-$current  on $\tilde M$ such that  $\tilde T \in f^* \{T\} = 0 \in H_{\de + \db}^{1,1}(\tilde M, \R):$
 this means that $\tilde T = \de  \overline s + \db s = 0$, since $\tilde M$ is \nkk. 
 
 Therefore $T = f_* \tilde T = 0.$
 \end{proof}

\medskip

Example 7.3 shows that similar results cannot hold for a generic $p$, also when  the exceptional set is supposed to be $pK$ as the manifold $M$, and the modification is simply a blow-up. Hence, to study when  a generalized $p-$K\"ahler property goes back from $M$ to $\tilde M$, we must add some hypothesis on $E$, as in the following result.

Since we shall use only here forms and currents on a (singular) analytic subset (that is, the exceptional set $E$), we refer to \cite{B}, pp. 575-577 for definitions and details; here, for a $(p,p)-$form $\tilde \Omega$ on $\tilde M$, we indicate by $i_E^* \tilde \Omega > 0$ the fact that, for every positive current $t$ on $E$, it holds $(( i_E)_* t,\tilde \Omega) > 0$.
\medskip

\begin{prop} Let $M, \tilde M$ be compact $n-$dimensional manifolds, let $f : \tilde M \to M$ be a modification with center $Y$ and exceptional set $E$ (call $i_E : E \to \tilde M$ the inclusion); let $1 \leq p <n-1$ and suppose  $M$ is \pkk. If there is a $(p,p)-$form $\tilde \Omega$ on $\tilde M$ such that $i_E^* \tilde \Omega > 0$ and $f_*(d \tilde \Omega)$  (or $f_*(i \ddb \tilde \Omega)$ in case PL) is a smooth form,
then $\tilde M$ is \pkk too.
\end{prop}

\begin{proof}
Let $T \geq 0, T \neq 0$, be an \lq\lq exact\rq\rq current of bidimension $(p,p)$ on $\tilde M$.
Since $f_*T$ has the same properties on $M$, we get $f_*T=0$, which implies that {\it supp} $T \subseteq E$. By Theorem 1.24 in \cite{B}, there is a current $t$ on $E$ such that $T = ( i_E)_* t$;
thus $(T,\tilde \Omega) = (( i_E)_* t,\tilde \Omega) = (t, i_E^* \tilde \Omega) > 0.$

Arguing as in Proposition 4.2, since $f_*(d \tilde \Omega)$ is smooth and exact, we have a $(p,p)-$form $\Psi$ on $M$ such that $f_*(d \tilde \Omega) = d \Psi$; moreover, $f^* (d \Psi) = f^* (f_*(d \tilde \Omega)) = d \tilde \Omega$, since they are smooth forms, which coincide on $\tilde M - E$.
Therefore, when $T=dS$, 
$$(T,\tilde \Omega) = (dS,\tilde \Omega) = (S,d \tilde \Omega) =(S,f^* (d \Psi)) = (dS,f^*  \Psi) = (f_* T , \Psi) = 0.$$
When $T = \de  \overline S + \db S$, the proof is similar, since by dimensional reasons, 
$(\de  \overline S + \db S,\tilde \Omega) =(d(S + \overline S),\tilde \Omega)$. 

In the pPL case, we have only to replace the operator $d$ by the operator $\ddb$.
\end{proof}

\bigskip

\section{Link between \pkk forms on $M$ and $\tilde M$}

Notice that, using currents, in Theorem 4.1 we lose the connection between metrics on $M$ and $\tilde M$: nevertheless, we can prove the following link:

\begin{prop} Let $M, \tilde M$ be compact $n-$dimensional manifolds, let $f : \tilde M \to M$ be a modification. For every \nkk metric $h$ with form $\omega$ on $M$, there is an \nkk metric $\tilde h$ with form $\tilde \omega$ on $\tilde M$ such that $\omega^{n-1}$ and $f_* \tilde \omega^{n-1}$ are in the same (relevant) cohomology class.
\end{prop}

In case K, this is Corollary 4.9 in \cite{AB5}; we consider here  a more general context, namely, that of \pkk manifolds with $p > {\rm dim} Y$, non necessarily compact.

\begin{thm} Let  $f : \tilde M \to M$ be a proper modification with a compact center $Y$ and exceptional set $E$. Suppose  $\tilde M$ and $M$ are \pkk manifolds, with $p > dim Y$, having \pkk forms $\tilde \Omega$ and $\Omega$. Then there is a \pkk form $\Gamma$ on $\tilde M$ such that $f_* \Gamma$ is \lq\lq cohomologous\rq\rq to $\Omega$.
\end{thm}

\medskip
Here $f_* \Gamma$ is \lq\lq cohomologous\rq\rq to $\Omega$  means:
$\{f_* \Gamma\} = \{\Omega\} \in H_{\ddb}^{p,p}(M)$ in case K, $\{f_* \Gamma\} = \{\Omega\} \in H_{\de + \db}^{p,p}(M)$ in cases WK, S, PL. 
The case $pK$, with $M$ and $\tilde M$ compact manifolds, is proved  in \cite{AB5}, Theorem 4.8; as regards the case $(n-1)S$, see Theorem 1.2 in 
\cite{X}; we will prove here the general case. 
 \medskip
 
\begin{proof}  Our goal is to get, as in Theorem 3.1, a positive constant $c$ such that $\Gamma := f^* \Omega  + c \Theta$ is the required form, where $\Theta$ is null-cohomologous and is obtained by changing $f_* \tilde \Omega$.
\medskip

Let us recall the following classical result (see f.i. \cite{V2} p. 251):

{\bf Remark.} Let $Y$ be a $s-$dimensional compact analytic subset of $M$; $Y$ has a fundamental system of neighborhoods $\{U\}$ such that $H_{dR}^q (U, \R) =0$ for $q > 2s$, and, for every coherent sheaf $\F$, $H^q(U, \F) =0$ for $q > s$.
\medskip

In \cite{ABL} we studied the case of 1-convex manifolds, where the cohomology groups $H^q(U, \F)$ are finite dimensional when $q > 0$. We proved there the following result:
\medskip

\begin{thm} {\rm (\cite{ABL}, Theorem 2.4)} Let $M$ be a complex manifold, and let $\O^k$ be the sheaf of germs of holomorphic $k-$forms on $M$. Suppose $dim H^j(M, \O^k) < \infty \ \forall k \geq 0, \ \forall j \geq s.$
Then the cohomology groups $H_{\ddb}^{p,p}(M)$  and $H_{\de + \db}^{p,p}(M)$ are Hausdorff topological vector spaces for every $p \geq s$.
\end{thm}
\medskip

Adapting its proof, which is based on an accurate analysis of exact sequences of sheaves and cohomology groups, we get in our situation (where the cohomology groups vanish):
\medskip

{\bf Claim.} Let $Y$ be a $s-$dimensional compact analytic subset of $M$; $Y$ has a fundamental system of neighborhoods $\{U\}$ such that $H_{dR}^q (U, \R) =0$ for $q > 2s$, and, for every coherent sheaf $\F$, $H^q(U, \F) =0$ for $q > s$.
Thus $Y$ has a fundamental system of neighborhoods $\{U\}$ such that $H_{\ddb}^{p,p}(U) = 0, \ H_{\de + \db}^{p,p}(U) =0$ for $p > s$. 
\bigskip

To give a hint of the first step ($s=0, p=1$) of the proof of this Claim, let us consider $\H$, the sheaf of germs of real pluriharmonic functions, with the following well-known exact sequences of sheaves (see \cite{ABL}, p. 260):
$$ 0 \to \R \to \O \to \H \to 0$$
and
$$0 \to \H \to  {\E}^{0,0}_\R \to ({\E}^{1,0} \oplus {\E}^{0,1})_\R \to \dots .$$
From the second one we can compute $H_{\ddb}^{1,1}(U)$, so that $H_{\ddb}^{1,1}(U) \simeq H^1(U, \H)$.

From the first one, we get 
$$ \dots \to H^1(U, \O) \to H^1(U, \H) \to H^2(U, \R) \to \dots .$$
Thus, by the previous Remark,we get $H_{\ddb}^{1,1}(U) \simeq H^1(U, \H) = 0$.
\bigskip

Let us turn back to the proof of Theorem 5.2; choose $U$ as in the previous Claim, and consider $f_* \tilde \Omega$. While in case $pK$, $f_* \tilde \Omega$ is closed, so that by 
$H_{\ddb}^{p,p}(U) = 0$ we get $f_* \tilde \Omega = i \ddb R$ on $U$, in the other cases it holds $\ddb f_* \tilde \Omega =0$, so that thanks to 
$H_{\de + \db}^{p,p}(U) =0$ we get $f_* \tilde \Omega = \de \overline S + \db S$, for some $(p, p-1)-$current $S$ on $U$.
\medskip

Recall that cohomology classes can be represented by currents or by forms (see also the proof of Proposition 4.2): thus, since $f_* \tilde \Omega$ is smooth on $M-Y$, we get on $U-Y$:

a) $f_* \tilde \Omega = i \ddb \alpha$ for some real $(p-1,p-1)-$form $\alpha$ on $U-Y$ in case pK, and

b) $f_* \tilde \Omega = \de \overline \beta + \db \beta$ for some $(p,p-1)-$form $\beta$ on $U-Y$ in cases pWK, pS and pPL.
\medskip

{\bf Claim.} In the previous notation, on $U-Y$ we get, respectively:

a) $i \ddb (R- \alpha) =0$, thus $R- \alpha = \gamma + \de \overline C + \db C$, where $\gamma$ is a real $\ddb-$closed form and $C$ is a $(p-1,p-2)-$current; when $p=1$, $R- \alpha = \gamma $ are smooth functions;

b) $\de \overline{(S - \beta)} + \db (S- \beta) =0$, thus $S - \beta = \gamma + \de A + \db B$, where $\gamma$ is a $(p,p-1)-$form such that $\de \overline \gamma + \db \gamma =0$, $A$ is a real $(p-1,p-1)-$current, $B$ is a $(p,p-2)-$current.

c) when $p=1$, $\de \overline{(S - \beta)} + \db (S- \beta) =0$, thus $S - \beta = \gamma + \alpha + \de h$, where $\gamma$ is a $(1,0)-$form such that $\de \overline \gamma + \db \gamma =0$, $\alpha$ is a holomorphic 1-form, $h$ is a real distribution.
\medskip

{\it Proof of the Claim.} In case a), $\gamma$ is a smooth representative of the class $\{ R- \alpha\} \in H_{\de + \db}^{p-1,p-1}(U-Y)$; when $p=1$, 
$R- \alpha$ itself is smooth.
\medskip

In case b), for $p > 1$, the proof is more involved: we can use exact sequences of sheaves and their cohomology groups as done in \cite{ABL} (see the proof of Proposition 2.2 there). In particular, let us consider
$$ \dots  (\E^{p,p-2} \oplus \E^{p-1,p-1} \oplus \E^{p-2,p})_{\R} \to (\E^{p,p-1} \oplus \E^{p-1,p})_{\R} \to \E^{p,p}_{\R} \to \E^{p+1,p+1}_{\R}  \dots$$
where the maps are, respectively,

$\sigma_{2p-2}(\zeta, \eta, \overline \zeta) = (\db \zeta + \de \eta, \db \eta + \de \overline \zeta), \ \ \sigma_{2p-1}(\varphi, \overline \varphi) = (\db \varphi + \ \de \overline \varphi), \ \ \sigma_{2p} = i \ddb$.
\medskip

Notice that $H_{\de + \db}^{p-1,p-1}(U-Y)$ is given by $\frac{Ker \sigma_{2p} }{Im \sigma_{2p-1}}$ on $U-Y$, but here we need $\frac{Ker \sigma_{2p-1} }{Im \sigma_{2p-2}}$ on $U-Y$.

Since
 $\de \overline{(S - \beta)} + \db (S- \beta) =0$, i.e. 
$\sigma_{2p-1}(S - \beta, \overline{S - \beta}) = 0$, it represents a class in  $\frac{Ker \sigma_{2p-1} }{Im \sigma_{2p-2}}$ on $U-Y$.
Choose a smooth representative of this class: this means precisely $S - \beta = \gamma + \de A + \db B$, as stated in the Claim. 
\medskip

In case c), when $p = 1$, the exact sequence of sheaves is the following
$$ \dots  ( \O^1 \oplus \E^{0,0} \oplus \overline \O^1)_{\R} \to (\E^{1,0} \oplus \E^{0,1})_{\R} \to \E^{1,1}_{\R} \to \E^{2,2}_{\R}  \dots$$
where the maps are, respectively,

$\sigma_{0}(\alpha, h, \overline \alpha) = (\alpha + \de h, \db h + \overline \alpha), \ \ \sigma_{1}(\varphi, \overline \varphi) = (\db \varphi + \ \de \overline \varphi), \ \ \sigma_{2} = i \ddb$.
\medskip

Since
 $\de \overline{(S - \beta)} + \db (S- \beta) =0$, i.e. 
$\sigma_{1}(S - \beta, \overline{S - \beta}) = 0$, it represents a class in  $\frac{Ker \sigma_{1} }{Im \sigma_{0}}$ on $U-Y$.
Choose a smooth representative of this class: this means precisely $S - \beta = \gamma + \alpha + \de h$, as stated in the Claim. 
\bigskip

Going back to the proof of Theorem 5.2, choose a neighborhood  $W \subset \subset U$ of $Y$, and take a cut-off function $\chi \in C_0^{\infty} (U)$, $\chi = 1$ on $W$. Define:

$D := \chi (\alpha + \gamma) + \de (\chi \overline C) + \db (\chi C)$, 

$F := \chi (\beta + \gamma) + \de (\chi A) + \db (\chi B)$ when $p>1$, and 

$F := \chi (\beta + \gamma + \alpha) + \de (\chi h)$ when $p=1$.

$D$ and $F$ are currents on $M-Y$; moreover, it is easy to check that $i \ddb D$ and $\de \overline F + \db F$ are smooth on $M-Y$, so that we can pull them back to $\tilde M - E$ (let us denote by $g$ the restriction of $f$ to $\tilde M - E$). Thus 
$\Theta := g^* (i \ddb D)$ and $\Theta ' := g^* (\de \overline F + \db F)$ are, respectively, $(p,p)-$forms on $\tilde M - E$, which coincide with $\tilde \Omega$ on $f^{-1} (W) - E$.

So they extend to the whole of $\tilde M$: remark that they are supported on $f^{-1}(U)$ and transverse on $f^{-1}(W)$.

Thus we can pick $c > 0$ such that $\Gamma := f^* \Omega + c \Theta$ or $\Gamma ' := f^* \Omega + c \Theta '$ are transverse forms on $\tilde M$.
$\Gamma$ and $\Gamma '$ are \lq\lq closed\rq\rq because $\Omega$ is  \lq\lq closed\rq\rq, and 

a) $\Theta$ is $i \ddb-$exact on $\tilde M - E$ and coincides with $\tilde \Omega$ (which is \lq\lq closed\rq\rq) near $E$;

b) $\Theta '$ is $(\de + \db)-$exact on $\tilde M - E$ and coincides with $\tilde \Omega$ (which is \lq\lq closed\rq\rq) near $E$.

Moreover, in the first case $f_* \Gamma - \Omega = i \ddb (cD)$, and in the other case 
$f_* \Gamma ' - \Omega ' = \de (\overline{cF}) + \db (cF)$. 
\end{proof}
\bigskip

\section{Currents in the non-compact case}

In the non-compact case, we cannot use the classical characterization of K\"ahler ma\-nifolds by currents, which has been introduced by Sullivan \cite{Su} 
and by Harvey and Lawson \cite{HL}; hence we have no results about a generic modification. Nevertheless, in \cite{ABL} we studied 1-convex manifolds (which are not compact but have a specific compact \lq\lq soul\rq\rq) by positive currents. This tecnique can be used to get a partial result on proper modifications.
Thus we consider the following definition, where $M$ is a complex $n-$dimensional manifold.

\begin{defn} Let $Y$ be a compact analytic subset of $M$; $M$ is said {\it locally \pkk} {\it with respect to $Y$} if every neighborhood $U$ of $Y$, $U \subset \subset M,$ is \pkk, in the sense that there is a real closed $(p,p)-$form $\Omega$ on $M$ such that $\Omega >0$ on the compact set  $\overline U$.
\end{defn}
\medskip

Our aim is to prove:

\begin{thm} Let $M, \tilde M$ be  $n-$dimensional manifolds, let $f : \tilde M \to M$ be a proper modification with compact center $Y$ and exceptional set $E$ (whose $(n-1)-$dimensional irreducible components are $\{E_j\}$). Suppose $dim H^j(\tilde M, \O^r) < \infty \ \forall r \geq 0, \ \forall j \geq 1.$
If  $M$ is locally $(n-1)$K with respect to $Y$, then $\tilde M$ is locally $(n-1)$K with respect to $E$.
\end{thm}
\medskip

Let us recall the notation and some results proved in \cite{ABL}, which we shall use in the proof. For every manifold $X$, $P^p_c(X)$ denotes the closed convex cone of positive currents of bidimension $(p,p)$ and compact support, while $B^p_c(X)$ is the space of currents of bidimension $(p,p)$ and compact support, which are $(p,p)-$components of a compactly supported boundary current, that is, the class of the current vanishes in 
$$H_{\de + \db}^{k,k}(X, \R)_c = \frac{\{T
\in {\E}_{p,p}'(X)_{\R}; i\ddb T =0\}}{\{ \de \overline S + \db S ; S \in  {\E}_{p,p+1}'(X)\}}.$$
 
\medskip

In the non-compact case, it is not guaranteed that the operators we need are topological homomorphisms (so that the orthogonal space to the Kernel coincides with the Image), but to get this fact it suffices to require a mild cohomological condition, as stated in the next result:

\begin{prop} {\rm (Corollary 2.5 in \cite{ABL})} Let $X$ be a complex manifold such that 

$dim H^j(X, \O^r) < \infty \ \forall r \geq 0, \ \forall j \geq 1.$ 
Then $d_p := d : \E^{p,p}_{\R}(X) \to (\E^{p+1,p} \oplus \E^{p,p+1})_{\R}(X)$  and $\ddb_p := \ddb : \E^{p-1,p-1}_{\R}(X) \to \E^{p,p}_{\R}(X)$
are topological homomorphisms for every $p \geq 1$.
\end{prop}
\medskip

So we got:

\begin{thm} {\rm (see Theorem 3.2 in \cite{ABL})} Let $X$ be a complex manifold, $K$ a compact subset of $X$; let $1 \leq p \leq n-1$ and suppose $d_p := d : \E^{p,p}_{\R}(X) \to (\E^{p+1,p} \oplus \E^{p,p+1})_{\R}(X)$ is a topological homomorphism. Then:

there is no current $T \neq 0$, $T \in P^p_c(X) \cap B^p_c(X),\  supp T \subseteq K \ \iff $ there is a real closed $(p,p)-$form $\Omega$ on $M$ such that $\Omega >0$ on $K$.
\end{thm}
\medskip
Now we can prove Theorem 6.2.

\begin{proof} Fix a neighborhood $U$ of $E$ in $\tilde M$, $U \subset \subset \tilde M$, and let $T$ be a {\it bad} current, i.e. $T \in P^p_c(\tilde M) \cap B^p_c(\tilde M),\ $  supp $T \subseteq K:= \overline U$. Thus $f_*T$ is a {\it bad} current on $M$ supported in $f(K)$, which is a compact neighborhood of $Y$: since $M$ is locally $(n-1)-$K\"ahler with respect to $Y$, we get $f_*T=0$, so that supp $T \subseteq E$.

By the Support Theorem 2.5, $T$ is closed (in fact, $T= \sum c_j [E_j], \ c_j \geq 0$), and moreover $T = \de  \overline S + \db S$ for some compactly supported  $(1,0)-$current $S$, so that we get $\ddb S =0$.
Consider $\de S$: it is a $\db -$closed $(2,0)-$current, hence it is a holomorphic 2-form with compact support: therefore,  $\de S=0$ and $T= d(S + \overline S)$ is $d-$exact. But no $(n-1)-$dimensional component of $E$ is null-homologous, by the structure of the homology of $\tilde M$: hence $T= \sum c_j [E_j] =0$. 
\end{proof}
\medskip

The other \pkk cases are not known. 

Notice that when $\sigma_{2p} := i \ddb : \E^{p,p}_{\R} \to \E^{p+1,p+1}_{\R}$ is a topological homomorphism (which is true in our hypothesis by Corollary 2.5 in \cite{ABL}, see above), then we have a similar characterization Theorem (see \cite{ABL}, Remark 3.4):
\medskip

\begin{prop} Let $X$ be a complex manifold, $K$ a compact subset of $X$; let $1 \leq p \leq n-1$ and suppose $\sigma_{2p} := i \ddb : \E^{p,p}_{\R} \to \E^{p+1,p+1}_{\R}$ is a topological homomorphism.Then:

there is no current $T \neq 0$, $T \in P^p_c(X) \cap (Im \sigma_{2p})_c ,\  supp T \subseteq K \ \iff $ there is a real  $(p,p)-$form $\Omega$ on $M$ such that $i \ddb \Omega=0$ and $\Omega >0$ on $K$.
\end{prop}
\medskip

But when $p=n-1$, $T \in P^p_c(X) \cap (Im \sigma_{2p})_c$ means that $0 < T= i \ddb g$, with $g$ a plurisubharmonic function with compact support: so $g$ is a constant, and $T=0$. This means that every $n-$dimensional complex manifold is $(n-1)$PL with respect to its compact subsets (as expected).
\bigskip

\section{Examples and remarks}

{\bf Examples.} 

{\bf 7.1} Hironaka's manifold $X$ (see \cite{Ha}, p. 444 or \cite {AB3}) is given by a modification $f : X \to \P_3$, where the center $Y$ is a plane curve with a node. It is a Moishezon manifold, containing a null-homologous curve. Thus it is not \kk. $X$ is a balanced manifold (see \cite{AB3}) so that it is \nkk. 

We can also consider a modification given as follows:  take $\pi_O$, the blow-up of $\P_3$ at a point $O$, and then perform a modification $f$ like that of Hironaka, where the center $Y$ of $f$ lies on the exceptional set of $\pi_O$. Here, the exceptional set $O$ is a point, but the resulting compact threefold is not \kk.

 \medskip

{\bf 7.2} In \cite{AB4} we build an example to show that, even in the case of modifications, we can sometime pull-back  \pkk properties for $p>1$. Indeed, we consider  a smooth modification $\tilde X$ of $\P_5$, where the center $Y$ is a surface with a singularity; the singular fibre has two irreducible components, one of which is biholomorphic to $\P_2$ and the other is a holomorphic fibre bundle over $\P_1$ with $\P_2$ as fibre. We show that $\tilde X$ is not K\"ahler, because it contains a copy of Hironaka's manifold, but it is \pkk for every $p > 1$.
\medskip

{\bf 7.3} On the contrary, we give here an example that shows that we cannot {\it always} pull-back \pkk properties by blowing up, even in the compact case. We use the class of manifolds we constructed in \cite{AB1}, namely, the compact nilmanifolds $\eta \beta _{2n+1}, \ n \geq 1$: let us recall the definition.

Let $G$ be the following subgroup of $GL(n+2, \C)$:
$$G := \{ A \in GL(n+2, \C) / A=\left(\begin{array}{ccc} 1 & X & z \\ 0 & I_n & Y \\ 0 & 0 & 1 \end{array} \right)
,  z \in \C, X, Y \in \C^n \},$$
and let $\Gamma$ be the subgroup of $G$ given by matrices with entries in $\Z[i]$. $\Gamma$ is a discrete subgroup and the homogeneous manifold 
$\eta \beta _{2n+1} := G/ \Gamma$ becomes a holomorphically parallelizable compact connected complex manifold (in particular, a nilmanifold) of dimensione $2n+1$ (for $n=1$, $\eta \beta _{3}$ is nothing but the Iwasawa manifold $I_3$). The standard basis for holomorphic 1-forms on $\eta \beta _{2n+1}$ is $\{ \varphi_1, \dots, \varphi_{2n+1} \}$; the  $ \varphi_j$ are all closed, except $d \varphi_{2n+1}  = \varphi_{1}  \wedge \varphi_{2}  + \dots + \varphi_{2n-1}  \wedge \varphi_{2n} .$

Recall the following results:

\begin{thm} {\rm (see Theorem 3.2 and Theorem 4.2 in \cite{AB1})}  (1) For $\eta \beta _{2n+1}$, as for all holomorphically parallelizable manifolds, 
for a fixed $p$, all \pkk conditions are equivalent.

(2) The manifold $\eta \beta _{2n+1}$ is not pK for $1 \leq p \leq n$ and is pK for $n+1 \leq p \leq 2n$.
\end{thm}

\medskip
To build our example, let us consider $M = \eta \beta _{7}, Y = \eta \beta _{3} = I_3$ as a submanifold of $M$ (in an obvious way, see f.i. (4.4) in \cite{AB1}). In particular, $M$ is 4K, Y is 2K but not K\"ahler. Consider $\pi : \tilde M \to M$, the blow-up of $M$ along $Y$; if $\tilde M$ were 4K too, then also the exceptional set $E$ would be 4K, but by definition $\pi$ induces a holomorphic submersion  from $E$ to $Y$ with 3-dimensional fibres. Thus $Y$ would be K\"ahler, by using the push-forward of a 4K form on $E$ (see Proposition 3.1 in \cite{A3}).
\medskip

{\bf 7.4} Taking into account these examples, let us collect what we got until now for the case of a modification of a {\it compact} \kk manifold $M$: 

a) $\tilde M$ is obviously $(n-1)$PL.

b) If $M$ is K\"ahler (i.e. 1K), then it is regular (in the sense of Varouchas, that is, it satisfies the $\ddb-$Lemma, see \cite{V3}, \cite{DGMS}), so that also $\tilde M$ is regular, which implies that it is $(n-1)$WK and $(n-1)$S by the following result stated in \cite{A1}: {\it On a regular manifold, $\forall \ p$,  $pWK = pS = pPL$. Thus, every regular manifold is $(n-1)WK $, since $(n-1)WK = (n-1)PL$.}

As stated in Theorem 4.1, $\tilde M$ is also $(n-1)$K (a direct proof was given in \cite{AB4}). Nevertheless, $\tilde M$ may not be \kk, as examples in  7.1 show, also when the center is only a point.
But example 7.2 shows that $\tilde M$ can be \pkk for every $p>1$.

c) If $M$ is \kk, then in case K and S it is also \nkk, so that $\tilde M$ is $(n-1)$K or, respectively, $(n-1)$S (see f.i. \cite{A3}). We don't know in general if,
when 
$M$ is 1WK, then  $\tilde M$ is $(n-1)$WK; this is true for a wide class of manifolds, for instance when $H^{2,0}(M) =0$, because in this case $(n-1)$WK = $(n-1)$S (see \cite{A1}).
\bigskip

{\bf 7.5} We recall here an example proposed by Yachou \cite{Y}, which illustrates the following result (see \cite{GR1}, pp. 506-507): If $G$ is a complex connected semisimple Lie group, it has a discrete subgroup $\Gamma$ such that the homogeneous manifold $M := G/\Gamma$ is compact, holomorphically parallelizable and has no hypersurfaces (since $a(M) = 0$).

Take $G=SL(2,\C)$, and let us consider the holomorphic $1-$forms $\eta, \alpha, \beta$ on $M := G/ \Gamma$ induced by the standard basis for $\g^*$: it holds
$$ d \alpha = -2 \eta \wedge \alpha, \ \ d \beta = 2 \eta \wedge \beta, \ \ d \eta = \alpha \wedge \beta.$$
The standard fundamental  form, given by $\omega = \frac{i}{2} (\alpha \wedge \overline \alpha + \beta \wedge \overline \beta +\eta \wedge \overline \eta )$, satisfies $d \omega^2=0$, so that $\omega^2$ is a balanced form: but it is exact, since
$$\omega^2 = d(\frac{1}{16} \alpha \wedge d \overline  \alpha + \frac{1}{16} \beta \wedge d \overline  \beta +\frac{1}{4} \eta \wedge d \overline \eta ).$$

\bigskip
{\bf 7.6}  Let us end with a particular question, related to example 7.5, i.e. the fact that, on compact K\"ahler manifolds, $\int_M \omega^n = $ vol $M > 0$, so that $\omega$ is not \lq\lq exact\rq\rq, while a 2K form can be exact, as seen in 7.5 (see also the Introduction of \cite{FX}). Notice that, when $M$ is not compact, a \pkk form can be exact: this is the case, for instance, of $p-$complete manifolds (see \cite{ABL}, Proposition 4.4). \medskip

Suppose $M, \tilde M$ are complex manifolds and $f: \tilde M \to M$ is a proper modification with compact center $Y$; suppose moreover that $\tilde M$ is \pkk: can the class of its \pkk form $\tilde \Omega$ vanish, in one of the following  cohomology groups: $H_{\ddb}^{p,p}(\tilde M),$ $H_{\de + \db}^{p,p}(\tilde M),$ $H_{dR} ^{p,p}(\tilde M) ?$
\medskip

In most cases, the answer is no:  suppose ${\rm dim} Y =s, 0 \leq s \leq n-2$: the cases $p=n-1$ and $p=n-1-s$ are completely solved by the existence of compact analytic subvarieties of the right dimension in $\tilde M$, namely, the maximal irreducible components of $E$ and the generic fibre $f^{-1}(y)$. These are closed currents, which vanish when applied to an \lq\lq exact\rq\rq  form, but they must be positive when applied to transverse forms, since they have positive volume. For the same reason, the answer is the same, for every $p$, on blow-ups with center at a point $O$: they have enough compact subvarieties on $E$.
\medskip

In some other cases, when $M$ and $\tilde M$ are compact, we can use the pull-back of a suitable $p-$K\"ahler form on $M$: in particular, this holds when $M$ is balanced, as follows:

\begin{prop} Let $f: \tilde M \to M$ be a modification, $M$ a compact balanced manifold with form $\omega$, $\tilde M$ a compact \kk manifold with \kk form $\tilde \omega$. Then $\tilde \omega$ is never exact.
\end{prop}

{\it Proof.} Case in $H_{\ddb}^{p,p}(\tilde M)$ is obvious, because $\tilde \omega = i \ddb g > 0$ implies $g$ is constant.

In the case 1S, if $\tilde \omega= \psi^{1,1}$ with $d \psi =0$, we ask for the possibility $\psi = d \alpha$. Notice that $f^*\omega^{n-1} \geq 0$ and 
$f^*\omega^{n-1} > 0$ outside $E$; thus we get
$$0 < \int_{\tilde M} \tilde \omega \wedge f^*\omega^{n-1} = \int_{\tilde M} \psi \wedge f^*\omega^{n-1} = \int_{\tilde M} d \alpha \wedge f^*\omega^{n-1} =
 - \int_{\tilde M}  \alpha \wedge d (f^*\omega^{n-1}),$$
 which vanishes since $M$ is balanced.
 
 In case 1WK and 1PL, starting by $\de \tilde \omega= \ddb \alpha$ or $\ddb \tilde \omega= 0$, we ask for the possibility $\tilde \omega= \de \overline \mu + \db \mu$; this can be solved as above.
 
 In case 1K, starting by $d \tilde \omega= 0$, we can ask if $\tilde \omega= d \beta$. As above, the answer is negative, also when $M$ is only $(n-1)$WK.
 \medskip
 
 {\bf Claim.}  Arguing as in the previous Proposition, if $M$ is compact K\"ahler and $\tilde M$ is \pkk, its form cannot be exact.

\bigskip
\bigskip

\end{document}